\pdfoutput=1
\RequirePackage{ifpdf}
\ifpdf % We are running pdfTeX in pdf mode
\documentclass[pdftex]{sigma}
\else
\documentclass{sigma}
\fi

\newcommand\bC{\mathbb C}
\newcommand\bH{\mathbb H}
\newcommand\bO{\mathbb O}
\newcommand\bR{\mathbb R}
\newcommand\bZ{\mathbb Z}

\newcommand\cO{\mathcal O}
\newcommand\cT{\mathcal T}
\newcommand\cW{\mathcal W}

\newcommand\mono\hookrightarrow
\newcommand\epi\twoheadrightarrow

\newcommand\<\langle
\renewcommand\>\rangle

\newcommand\so{\mathfrak{so}}
\newcommand\Cliff{\mathrm{Clif\/f}}
\newcommand\Spin{\mathrm{Spin}}

\newcommand\g{\mathfrak{g}}

\renewcommand\d{\mathrm{d}}

\DeclareMathOperator\End{End}

\begin{document}

%\allowdisplaybreaks

\renewcommand{\thefootnote}{$\star$}

\newcommand{\arXivNumber}{1603.06603}

\renewcommand{\PaperNumber}{116}

\FirstPageHeading

\ShortArticleName{The Quaternions and Bott Periodicity are Quantum Hamiltonian Reductions}

\ArticleName{The Quaternions and Bott Periodicity\\ Are Quantum Hamiltonian Reductions\footnote{This paper is a~contribution to the Special Issue ``Gone Fishing''. The full collection is available at \href{http://www.emis.de/journals/SIGMA/gone-fishing2016.html}{http://www.emis.de/journals/SIGMA/gone-fishing2016.html}}}

\Author{Theo JOHNSON-FREYD}

\AuthorNameForHeading{T.~Johnson-Freyd}

\Address{Perimeter Institute for Theoretical Physics, Waterloo, ON, Canada}
\Email{\href{mailto:theojf@perimeterinstitute.ca}{theojf@perimeterinstitute.ca}}
\URLaddress{\url{https://perimeterinstitute.ca/personal/tjohnsonfreyd/}}

\ArticleDates{Received August 30, 2016, in f\/inal form December 09, 2016; Published online December 11, 2016}

\Abstract{We show that the Morita equivalences $\mathrm{Clif\/f}(4) \simeq {\mathbb H}$, $\mathrm{Clif\/f}(7) \simeq \mathrm{Clif\/f}(-1)$, and $\mathrm{Clif\/f}(8) \simeq {\mathbb R}$ arise from quantizing the Hamiltonian reductions ${\mathbb R}^{0|4} // \mathrm{Spin}(3)$, ${\mathbb R}^{0|7} // G_2$, and ${\mathbb R}^{0|8} // \mathrm{Spin}(7)$, respectively.}


\Classification{15A66; 53D20; 16D90; 81Q60}

\renewcommand{\thefootnote}{\arabic{footnote}}
\setcounter{footnote}{0}

\medskip

This note provides (super) symplectic origins for the quaternion algebra $\bH$ and for the eight-fold ``Bott periodicity'' of Clif\/ford algebras (due originally to Cartan \cite{CartanNombresComplexes}) in terms of quantum Hamiltonian reduction. Clif\/ford algebras arise in symplectic supergeometry as the Weyl (aka canonical commutation) algebras of purely-odd symplectic supermanifolds $\bR^{0|n}$. As we explain, Hamiltonian reductions quantize to bimodules, which are often Morita equivalences. In particular, we will show that the well-known Morita equivalence $\bH\simeq \Cliff(4)$ is the quantization of the Hamiltonian reduction $\bR^{0|4} // \Spin(3)$, where $\Spin(3) = \mathrm{SU}(2)$ acts on $\bR^{0|4}$ as the underlying real module of the def\/ining action of $\mathrm{SU}(2)$ on $\bC^2$, and that the reduction $\bR^{0|8} // \Spin(7)$ coming from the spin representation quantizes to the ``Bott periodicity'' Morita equivalence $\Cliff(8) \simeq \bR$. We also show that the Morita equivalence $\Cliff(7) \simeq \Cliff(-1)$ arises from the Hamiltonian reduction $\bR^{0|7} // G_2$, where $G_2 \subseteq \mathrm{SO}(7)$ is the exceptional Lie group of automorphisms of the octonion algebra $\bO$.

\section{Symplectic supermanifolds and Clif\/ford algebras} \label{section super}

A \textit{superalgebra} is a $\bZ/2$-graded associative algebra (meaning, in particular, that the multiplication adds degree modulo $2$); morphisms are grading-preserving. A~\textit{supermodule} is a $\bZ/2$-graded module. If $M$ is a left $A$-supermodule, the algebra $\End_A(M)$ of all $A$-linear endomorphisms of~$M$ is naturally a superalgebra acting on $M$ from the right (with multiplication $fg = g\circ f$). Two superalgebras $A$ and $B$ are \textit{super Morita equivalent} if there are $\bZ/2$-graded bimodules~$_AM_B$ and~$_BN_A$ with grading-preserving bimodule isomorphisms $M\otimes_B N \cong A$ and $N \otimes_A M \cong B$. We will generally suppress the word ``super'': for example, ``module'' and ``Morita equivalence'' will henceforth always be meant in the super sense.

A superalgebra $A$ is \textit{commutative} if for homogeneous elements $x$ and $y$ (with degrees~$|x|$ and~$|y|$), $yx = (-1)^{|x|\cdot|y|}xy$. Note in particular that for an odd element $x$ in a commutative superalgebra, $x^2 = - x^2$, and so $x^2=0$. By def\/inition, \textit{odd $n$-dimensional space} $\bR^{0|n}$ is the ``spectrum'' of the commutative superalgebra $\cO(\bR^{0|n}) = \bR[x_1,\dots,x_n]$ where the \textit{coordinate functions} $x_1,\dots,x_n$ are odd, and $\bR[\dots]$ denotes the free commutative superalgebra on ``$\dots$''. Thus $\cO(\bR^{0|n}) \cong \bigwedge^\bullet \bR^n$. In general, a \textit{supermanifold} is a ``space'' that looks locally like $\bR^{m|n} = \bR^m \times \bR^{0|n}$. See~\cite{DM1999} for details on superalgebras and supermanifolds.

\looseness=-1 A \textit{symplectic structure} on a supermanifold is an even nondegenerate closed de Rham 2-form. De Rham forms can be def\/ined for commutative superalgebras just like for commutative algebras, but behave dif\/ferently in one important way: if $x$ is an odd coordinate, then $\d x$ is even, and so $\d x \wedge \d x \neq 0$, and if $x$ and $y$ are both odd, then $\d x \wedge \d y = \d y \wedge \d x$ with no sign. A side ef\/fect of this is that symplectic structures on odd manifolds behave somewhat like metrics on even manifolds.

Equip $\bR^{0|n}$ with the positive-def\/inite symplectic form $\omega = \sum_i \frac{(\d x_i)^2}2$. The corresponding Poisson structure on $\bR^{0|n}$ is given by the Poisson brackets $\{x_i,x_j\} = -2\delta_{ij}$. (The sign depends on an essentially-arbitrary choice of convention for inverse matrices in superalgebra.) The symplectic form $\omega$ on $\bR^{0|n}$ is translation-invariant and so admits a \textit{canonical quantization} to the Weyl algebra $\cW(\bR^{0|n}) = \bR\langle x_1,\dots,x_n\rangle / ([x_i,x_j] = \{x_i,x_j\})$, where by def\/inition in a superalgebra the {commutator} is def\/ined on homogeneous elements by $[x,y] = xy - (-1)^{|x|\cdot|y|}yx$. Thus~$\cW(\bR^{0|n})$ is the \textit{Clifford algebra} $\Cliff(n) = \bR\langle x_1,\dots,x_n\rangle / (x_i^2 = -1, x_ix_j = -x_jx_i \text{ for }i\neq j)$ with its usual $\bZ/2$-grading in which all generators $x_i$ are odd.
The Weyl algebra of $\bR^{0|n}$ equipped with symplectic form $-\omega$ is $\Cliff(-n) = \bR\langle x_1,\dots,x_n\rangle / (x_i^2 = 1, x_ix_j = -x_jx_i)$.

\section{Quantum Hamiltonian reduction} \label{section reduction}

A \textit{moment map} for the action of a super Lie group $G$ on a symplectic supermanifold $M$ is a~map $\mu \colon M \to \g^* = \mathrm{Lie}(G)^*$ of Poisson supermanifolds such that the inf\/initesimal action of an element $a\in \g$ is given by the Hamiltonian vector f\/ield for the function $m \mapsto \langle \mu(m),a\rangle$, where~$\langle\,,\,\rangle$ denotes the pairing of the vector space $\g$ with its dual; such data is equivalent to a Lie algebra map $\mu^*\colon \g \to \cO(M)$, where the latter is treated as a super Lie algebra with its Poisson bracket. When certain cohomology groups of $M$ and $\g$ vanish, $\mu$ exists and is unique. The \textit{Hamiltonian reduction} $M // G$ of this data is the quotient space $\mu^{-1}(0) / G$. This can be def\/ined in the super case via its algebra of functions $(\cO(M)/\langle \mu^*\g\rangle)^G$, where $\langle \mu^*\g\rangle$ denotes the ideal generated by the image of $\mu^*$. As Marsden and Weinstein explained in the even case~\cite{MR0402819}, when $0$ is a regular value of $\mu$ and the action of $G$ on $\mu^{-1}(0)$ is free and proper, the manifold $M // G$ is naturally symplectic. The natural maps $\mu^{-1}(0) \mono M$ and $\mu^{-1}(0) \epi M//G$ are together a Lagrangian correspondence between $M$ and $M//G$. Super Hamiltonian reduction can be cleanly expressed as an example of coisotropic reduction of super Poisson algebras \cite{MR3027582, MR2795841}.

\looseness=-1 Suppose that $G$ acts instead on an associative superalgebra $A$. A \textit{comoment map} is a Lie algebra map $\mu^* \colon \g \to A$, where $A$ is treated as a super Lie algebra with its commutator bracket, such that the inf\/initesimal action of $a \in \g$ is given by the inner derivation $[\mu^*(a),-]$. Corresponding to the zero section $\mu^{-1}(0)$ is the quotient module $A / \langle \mu^*\g\rangle$, where $\langle \mu^*\g\rangle$ denotes the left ideal generated by the image of $\mu^*$. Corresponding to the quotient $M // G = \mu^{-1}(0)/G$ is the \textit{quantum Hamiltonian reduction} $A//G = (A/\langle \mu^*\g\rangle)^G$. This is naturally an algebra because it is isomorphic to $\End_A(A/\langle \mu^*\g\rangle) $. When $G$ is compact, $A//G \cong A^G / \bigl( A^G \cap \langle \mu^*\g\rangle\bigr)$, where $A^G$ denotes the $G$-invariant subalgebra of $A$. From this perspective, the algebra structure on $A//G$ arises because, although $\langle \mu^*\g\rangle$ is merely a left ideal in $A$, its intersection with $A^G$ is a two-sided ideal, as~$\mu^*\g$ is central in~$A^G$. The module $A / \langle\mu^*\g\rangle$ is by construction a bimodule between~$A$ and~$A//G$.

\begin{example}\label{linear lagrangians}
 Suppose that $M$ is a linear symplectic supermanifold and $C \subseteq M$ is a coisotropic submanifold cut out by linear equations $r_1 = \dots = r_{p+q} = 0$, where $r_1,\dots, r_p$ are even and $r_{p+1},\dots,r_{p+q}$ are odd. The Hamiltonian f\/lows for $r_1,\dots,r_{p+q}$ def\/ine an action on $M$ of the abelian Lie supergroup $\bR^{p|q}$. Let $C^\perp \subseteq C$ denote the symplectic orthogonal to~$C$. The Hamiltonian reduction $M // \bR^{p|q}$ is then canonically linearly symplectomorphic to $C/C^\perp$.

\looseness=-1 Since $M$ is linear, it admits a canonical quantization to the Weyl algebra $\cW(M) {=} \cT(M^*) {/} ([a{,}b]\!$ $= \{a,b\}, a,b\in M^*)$. The quotient $\cW(C) = \cW(M) / \langle r_1,\dots,r_c\rangle$ is the canonical quantization of~$C$, and $\cW(M) // \bR^{p|q} \cong \cW(C / C^\perp)$. In the purely-odd case, which is the only case of concern in this paper, $\cW(C)$ is a Morita equivalence between $\cW(M)$ and $\cW(C/C^\perp)$: it suf\/f\/ices to consider the case $M =\bR^{0|2}$ with ``split'' symplectic form $\frac{(\d x)^2}2 - \frac{(\d y)^2}2$ and Lagrangian $C \cong \bR^{0|1}$ spanned by the vector $(1,1)$; then $\cW(M) \cong \mathrm{Mat}(1|1)$ is the algebra of $2\times 2$ matrices in which $\bigl(\begin{smallmatrix} 1 & 0 \\ 0 & 0\end{smallmatrix}\bigr)$ and $\bigl(\begin{smallmatrix} 0 & 0 \\ 0 & 1\end{smallmatrix}\bigr)$ are even and $\bigl(\begin{smallmatrix} 0 & 1 \\ 0 & 0\end{smallmatrix}\bigr)$ and $\bigl(\begin{smallmatrix} 0 & 0 \\ 1 & 0\end{smallmatrix}\bigr)$ are odd, and $\cW(C)$ is the def\/ining $(1|1)$-dimensional module.

(When there are even coordinates, $\cW(C)$ is not a Morita equivalence. The Stone--von Neumann theorem can be understood as saying that for purely even $M$, $\cW(C)$ becomes a Morita equivalence after appropriate functional analytic completions. The mixed case can be handled by decomposing $M$ and $C$ into even and odd parts.)

In particular, linear Lagrangians provide Morita equivalences $\cW(M) \simeq \bR$. This does not explain why $\Cliff(8) = \cW(\bR^{0|8}) \simeq \bR$, because the positive-def\/initeness of the symplectic form prevents $\bR^{0|n}$ from admitting Lagrangian sub-supermanifolds, linear or not.
\end{example}

\begin{lemma}\label{non-zero suffices} If the Hamiltonian reduction $\Cliff(n) // G$ is not the zero algebra, then $\Cliff(n) / \langle \mu^*\g\rangle$ is a Morita equivalence between $\Cliff(n)$ and $\Cliff(n) // G$.
\end{lemma}

\begin{proof}For any superalgebra $A$, an $A$-module $X$ is a Morita equivalence between $A$ and $\End_A(X)$ if and only if $X$ is a f\/initely-generated projective generator of the supercategory of $A$-modules. The holomorphic symplectic supermanifold $\bC^{0|n} = \bR^{0|n} \otimes \bC$ admits a linear Lagrangian $L$ if $n$ is even and an $(n+1)/2$-dimensional linear coisotropic $C$ if $n$ is odd. Via Example~\ref{linear lagrangians}, these linear coisotropics provide Morita equivalences $\Cliff(n) \otimes \bC \simeq \bC = \cW(L/L^\perp)$ or $\Cliff(1) \otimes \bC = \cW(C/C^\perp)$. For $\bC$ and $\Cliff(1) \otimes \bC$, any non-zero f\/initely-generated module is a projective generator. But ``non-zero'', ``f\/initely-generated'', and ``projective'' are Morita-invariant notions, so these properties hold also for $\Cliff(n)\otimes \bC$ and hence for $\Cliff(n)$.
\end{proof}

\section[$\Cliff(4)$ and $\bH$]{$\boldsymbol{\Cliff(4)}$ and $\boldsymbol{\bH}$}

Corresponding to the exceptional isomorphism $\mathrm{SO}(4) \cong \Spin(3) \times_{\bZ/2} \Spin(3)$ are two commuting actions of $\Spin(3)$ on $\bR^{0|4}$ by linear symplectic automorphisms. (Odd symplectic groups are even orthogonal groups; metaplectic groups correspond to spin groups.) Denote the coordinates on $\bR^{0|4}$ by $\{w,x,y,z\}$ and the bases for two copies of $\so(3)$ by $\{a_+,b_+,c_+\}$ and $\{a_-,b_-,c_-\}$, normalized so that their brackets are $[a_\pm,b_\pm] = \pm 2c_\pm$, $[b_\pm,c_\pm] = \pm 2a_\pm$, $[c_\pm,a_\pm]=\pm 2b_\pm$. The comoment maps for the actions are:
\begin{gather*}
 a_\pm \mapsto \tfrac12(wx\pm yz), \qquad b_\pm \mapsto \tfrac12(wy\pm zx), \qquad c_\pm \mapsto \tfrac12(wz\pm xy).
\end{gather*}
Together these six elements are a basis for the space of homogeneous-quadratic functions on~$\bR^{0|4}$. The quadratic Casimir for both $\so(3)$s is $\theta = \pm 2a_\pm^2= \pm 2b_\pm^2=\pm 2c_\pm^2= wxyz$. Completing the basis for $\cO(\bR^{0|4})$ are the unit $1$ and $xyz$, $wyz$, $wzx$, and $wxy$. Basis vectors $1$, $a_\pm$, $b_\pm$, $c_\pm$, and~$\theta$ are even, and $x$, $y$, $z$, $w$, $xyz$, $wyz$, $wzx$, and $wxy$ are odd.

We now consider the quantization $\Cliff(4) = \cW(\bR^{0|4})$, for which we can use the same basis $\{1,w,x,y,z,a_+,b_+,c_+,a_-,b_-,c_-,xyz,wyz,wzx,wxy,\theta\}$ (with the same grading). Note that, whereas in $\cO(\bR^{0|4})$ we had $a^2_\pm = b^2_\pm = c^2_\pm =\pm \theta/2$, in $\Cliff(4)$ we have $ a^2_\pm = b^2_\pm = c^2_\pm = \frac12(\pm \theta-1)$. Since the actions of $\mathrm{Spin}(3)$ on $\bR^{0|4}$ are linear, they lift to $\Cliff(4)$, and the same assignments $a_+,\dots,c_-$ provide the quantum comoment maps.

\begin{theorem} \label{theoremH}
 The quantum Hamiltonian reduction of either of the $\so(3)$-actions on $\bR^{0|4}$ produces a Morita equivalence $\Cliff(4)\simeq \bH = \bR\langle i,j,k\rangle / (i^2 = j^2 = k^2 = ijk = -1)$ $($where the quaternion algebra $\bH$ is purely even$)$.
\end{theorem}

\begin{proof}
There is a manifest symmetry interchanging the two $\so(3)$-actions; we will work with the action of the $ a_-$, $b_-$, and $c_-$. It is not hard to see that
$[a_-,-]$, $[b_-,-]$, and $[c_-,-]$ preserve polynomial degree. We have observed already that the quadratic elements $a_+$, $b_+$, $c_+$ commute with the generators $a_-$, $b_-$, $c_-$ of the action, as well as with $\theta = wxyz$. The subspace of $\Cliff(4)$ spanned by $\{a_-,b_-,c_-\}$ is a submodule for the action of $\{a_-,b_-,c_-\}$ isomorphic to the adjoint action. The subspaces spanned by $\{w,x,y,z\}$ and $\{xyz,wyz,wzx,wxy\}$ are each isomorphic to the underlying real module of the def\/ining module of $\mathfrak{su}(2)$. Thus a basis for the $\so(3)$-f\/ixed subalgebra $\Cliff(4)^{\Spin(3)}$ is given by the f\/ive even elements $\{1,a_+,b_+,c_+,\theta\}$.

The left ideal $\langle \mu^*\so(3)\rangle$ in $\Cliff(4)$ generated by $\{a_-,b_-,c_-\}$ is eight-dimensional with basis $\{ a_-, b_-, c_-, w- xyz, x+wyz, y+wzx, z+wxy, \theta + 1\}$. This ideal intersects $\Cliff(4)^{\Spin(3)}$ only in the one-dimensional space spanned by $\theta + 1$. It follows that in the quotient $\Cliff(4)^{\Spin(3)}/$ $\bigl( \Cliff(4)^{\Spin(3)} \cap \langle \mu^*\so(3) \rangle\bigr)$ we have $a_+^2 = \frac12(\theta-1) \equiv \frac12(-2) = -1$, and so the map $(a_+,b_+,c_+) \mapsto (i,j,k)$ identif\/ies the quantum Hamiltonian reduction $\Cliff(4)//\Spin(3)$ with the quaternion algebra~$\bH$. Lemma~\ref{non-zero suffices} completes the proof.
\end{proof}

\looseness=-1 Theorem~\ref{theoremH} suggests that the ``classical limit'' of $\bH$ is the Hamiltonian reduction $\bR^{0|4} // \Spin(3)$. Since $0$ is not a regular value of the classical moment map, $\bR^{0|4} // \Spin(3)$ is not a supermanifold. It does make sense as an af\/f\/ine super scheme: its algebra of functions is the purely even Poisson algebra $\bR[a,b,c] / (a^2 = b^2 = c^2 = ab = bc= ca = 0)$ with Poisson brackets $\{a,b\} = 2c$, $\{b,c\} = 2a$, and $\{c,a\} = 2b$. Thus the ``classical limit'' of $\bH$ is the f\/irst-order neighborhood of~$0$ in~$\so(3)^*$.

\section[$\Cliff(7)$ and $G_2$]{$\boldsymbol{\Cliff(7)}$ and $\boldsymbol{G_2}$}

The 14-dimensional exceptional Lie group $G_2$ is the subgroup of $\mathrm{SO}(7)$ preserving the alter\-na\-ting 3-form $\epsilon$ on $\bR^7$ def\/ined by identifying $\bR^7$ with the pure-imaginary octonions and setting $\epsilon(a,b,c) = (ab)c - a(bc) \in \bR$~\cite{MR2282011}. Since $\mathrm{SO}(7)$ acts by linear symplectic automorphisms of $\bR^{0|7}$, we get an induced symplectic action of~$G_2$.

\begin{theorem}\label{theoremG}
 The quantum Hamiltonian reduction $\Cliff(7) // G_2$ provides the Morita equivalence $\Cliff(7) \simeq \Cliff(-1)$.
\end{theorem}

\begin{proof}
By Lemma~\ref{non-zero suffices}, it suf\/f\/ices to compute $\Cliff(7) // G_2$. As in Theorem~\ref{theoremH}, the Poincar\'e--Birkof\/f--Witt isomorphism $\cO(\bR^{0|7}) = \bigwedge^\bullet(\bR^7) \cong \Cliff(7)$ is $\mathrm{SO}(7)$-equivariant by the functoriality of the Weyl algebra construction. The $G_2$-f\/ixed algebra has as its basis a set of the form $\{1,\epsilon,\bar\epsilon,\theta\}$, where $\epsilon$ is the cubic 3-form def\/ining $G_2$, $\bar\epsilon$ is its dual quartic, and $\theta$ is the generator of $\bigwedge^7 \bR^7$. In $\Cliff(7)$, $\bar\epsilon = \theta\epsilon$ and $\theta^2=1$.

An explicit presentation of the action of $\mathrm{Lie}(G_2) = \g_2$ is given in \cite{ArenasG2} as follows. Denote the coordinates on $\bR^{0|7}$ by $x_1,x_2,\dots,x_7$. The cubic function $\epsilon$ is
\begin{gather*}
\epsilon = x_1x_2x_3 + x_1x_4x_5 + x_1x_6x_7 + x_2x_4x_6 + x_2x_7x_5 + x_3x_7x_4 + x_3x_6x_5.
\end{gather*}
Consider the quadratic functions $e_1,\dots,e_7$ def\/ined by $e_i = \frac{\partial}{\partial x_i}\epsilon$. For example, $e_1 = x_2x_3 + x_4x_5 + x_6x_7$.
 Orthogonal to each $e_i$ is a two-dimensional vector space of commuting quadratics given by the dif\/ferences of the monomials in $e_i$. For example, orthogonal to $e_1$ are $x_2x_3 - x_4x_5$, $x_4x_5 - x_6x_7$, and their sum $x_2x_3 - x_6x_7$. A basis for the image of $\g_2$ under $\mu^*$ is given by choosing for each $i = 1,\dots,7$ two quadratics orthogonal to $e_i$. For example:
 \begin{gather*}
 x_2x_3 - x_4x_5, \qquad x_4x_5 - x_6x_7, \qquad x_3x_1 - x_4x_6, \qquad x_4x_6 - x_7x_5, \qquad x_1x_2 - x_7x_4, \\
 x_7x_4 - x_6x_5, \qquad x_5x_1 - x_6x_2, \qquad x_6x_2 - x_3x_7, \qquad x_1x_4 - x_2x_7, \qquad x_2x_7 - x_3x_6, \\
 x_7x_1 - x_2x_4, \qquad x_2x_4 - x_5x_3, \qquad x_1x_6 - x_5x_2, \qquad x_5x_2 - x_4x_3.
 \end{gather*}
We wish to compute $\End_{\Cliff(7)}(\Cliff(7)/\langle \mu^*\g_2\rangle) = \Cliff(7)^{G_2} / (\langle \mu^*\g_2\rangle \cap \Cliff(7)^{G_2})$, where $\langle \mu^*\g_2\rangle$ is the left ideal generated by these $14$ elements.

Note that $x_1x_4x_5x_6x_7 (x_2x_3 - x_4x_5) = \theta + x_1x_6x_7$, where $\theta = x_1x_2\cdots x_7 \in \Cliff(7)^{G_2}$. The numerics of the second summand are: $x_2x_3-x_4x_5$ was orthogonal to $e_1$; $x_6x_7$ is the unused monomial in $e_1$. Similarly, for each monomial $\mu$ in the cubic $\epsilon$ one can f\/ind $\theta + \mu \in \langle \mu^*\g_2\rangle$, and summing shows that $7\theta + \epsilon \in \langle \mu^*\g_2\rangle \cap \Cliff(7)^{G_2}$; hence also $7+\bar\epsilon \in \langle \mu^*\g_2\rangle \cap \Cliff(7)^{G_2}$. It follows that $\Cliff(7)^{G_2} / (\langle \mu^*\g_2\rangle \cap \Cliff(7)^{G_2})$ is a quotient of the copy of $\Cliff(-1)$ spanned by the classes of $1$ and $\theta$.

Finally, for any basis element $\alpha \in \mu^*\g_2$, we have $\alpha(\theta - \epsilon) = 0$, from which it follows that $1 \not\in \langle \mu^*\g_2\rangle$. The ideal cannot mix even and odd terms without setting both to $0$, and so we f\/ind $\Cliff(7)^{G_2} / (\langle \mu^*\g_2\rangle \cap \Cliff(7)^{G_2}) \cong \Cliff(-1)$.
\end{proof}

\section[$\Spin(7)$ and Bott periodicity]{$\boldsymbol{\Spin(7)}$ and Bott periodicity}

{\sloppy We conclude by providing a Hamiltonian reduction whose quantization is the famous ``Bott pe\-rio\-dicity'' equivalence $\Cliff(8) \simeq \bR$. The irreducible real spin representations of all four groups $\Spin(5)$, $\Spin(6)$, $\Spin(7)$, and $\Spin(8)$ are eight-real-dimensional. The reduction $\Cliff(8) // \Spin(8)$ vanishes since the image of the comoment map consists of all quadratic elements of $\Cliff(8)$, including some which are invertible, and so the $\Spin(8)$-action does not induce a Morita equivalence. The reader is invited to compute $\Cliff(8) // \Spin(5)$ and $\Cliff(8) // \Spin(6)$. We will show:

}

\begin{theorem}\label{theorem8}
 $\Cliff(8) // \Spin(7) \cong \bR$.
\end{theorem}

By Lemma~\ref{non-zero suffices}, Theorem~\ref{theorem8} establishes that the cyclic module $\Cliff(8) / \langle \mu^*\so(7) \rangle$ is a Morita equivalence between $\Cliff(8)$ and $\bR$.

\begin{proof}
 The following construction of $\Spin(7)$, and its eight-dimensional spin representation, are developed in \cite{MR666108}. Consider the octonion algebra $\bO$ and the 4-form $\phi \in \bigwedge^4 \bO^*$ def\/ined by $\phi(a,b,c,d) = \langle a,b\times c \times d\rangle$, where the triple cross product is by def\/inition $b\times c \times d = \frac12 \bigl( b(\bar c d) - d(\bar c b) \bigr)$ and $\bar c$ is the octonionic conjugate of $c$. Then $\Spin(7)$ is precisely the subgroup of $\mathrm{SO}(8)$ f\/ixing $\phi$. In terms of coordinates $x_1,x_2,\dots,x_8$ on $\bR^{0|8}$, $\phi$ corresponds to the function
 \begin{gather*}
 \phi = x_{1234} + x_{1256} + x_{1278} + x_{1357} - x_{1368} - x_{1458} - x_{1467} \\
 \hphantom{\phi =}{} + x_{5678} + x_{3478} + x_{3456} + x_{2468} - x_{2457} - x_{2367} - x_{2358} ,
 \end{gather*}
 where we have abbreviated $x_{ij\dots k} {=} x_ix_j \cdots x_k$. Let $\theta {=} x_{12345678}$.
 The f\/ixed algebra $\Cliff(8)^{\Spin(7)}$ has basis $\{1,\phi,\theta\}$ and multiplication $\theta^2 = 1$, $\theta \phi = \phi\theta = \phi$, and $\phi^2 = 14\theta + 14 - 12\phi$.

 The image $\mu^*\so(7) \subseteq \bigwedge^2 \bR^8 \subseteq \Cliff(8)$ of the comoment map is spanned by quadratic elements of the form $\alpha = x_{ij} \pm x_{kl}$ such that $-\frac12 \alpha^2 - 1 = \mp x_{ijkl}$ is (with the given sign) a monomial in $\phi$. (There are $3 \cdot 14$ such $\alpha$s; given $\{i,j\}$, there are three $\alpha$s that include the monomial $x_{ij}$ and three that are disjoint from $\{i,j\}$, and these six span a three-dimensional space; this counts correctly the $21$-dimensional space $\so(7)$.) It follows that $\phi +14$, being a sum of $14$ terms of the form $-\frac12 \alpha^2$, is in the ideal $\langle \mu^* \so(7) \rangle$, and so $\Cliff(8) // \Spin(7)$ is a quotient of $\bR\{1,\phi,\theta\} / \langle \phi +14\rangle \cong \bR$.

 Each $\alpha = x_{ij} \pm x_{kl} \in \mu^*\so(7)$ determines a splitting $\phi = x_{ijkl}(1+\theta) + \kappa + \lambda$ where $\kappa$ is a~sum of four quartic monomials each of which has indices containing either $\{i,j\}$ or $\{k,l\}$ but not both, and $\lambda$ is a sum of eight quartic monomials each of which has indices intersecting the sets $\{i,j\}$ and $\{k,l\}$ at one element each. For example, when $\alpha = x_{13} - x_{57}$, we have
 \begin{gather*} \phi = \underbrace{x_{1357} + x_{2468}}_{x_{ijkl}(1+\theta)} + \underbrace{x_{1234} -x_{1368} + x_{5678} - x_{2457}}_\kappa \\
 \hphantom{\phi =}{} + \underbrace{x_{1256} + x_{1278} - x_{1458} - x_{1467} + x_{3478} + x_{3456}- x_{2367} - x_{2358}}_\lambda.
 \end{gather*}
 We see that $[\alpha,x_{ijkl}(1+\theta)] = 0$ and $ [\alpha,\kappa] = 0$. Since we know that $[\alpha,\phi] = 0$, we f\/ind $[\alpha,\lambda] = 0$ as well. Suppose $\beta$ is a quadratic monomial and $\nu$ a quartic monomial such that the indices in~$\beta$ and~$\nu$ overlap at one element. Then $\beta\nu = \frac12[\beta,\nu]$, from which it follows that $\alpha\lambda = \frac12[\alpha,\lambda] = 0$. It's also clear that $\alpha\kappa = 0$, since $\kappa$ factors as $(x_{ij} \mp x_{kl})(\dots)$ where $\alpha = x_{ij} \pm x_{kl}$ and $(\dots)$ is a~sum of two quadratic monomials that have no overlap with $\alpha$. Finally, $\alpha x_{ijkl}(1+\theta) = \alpha(-1-\theta)$. All together, we f\/ind:
\begin{gather*}
\alpha ( \phi + 1 + \theta) = \alpha\bigl( x_{ijkl}(1+\theta) + \kappa + \lambda + (1+\theta)\bigr) = \alpha(-1-\theta) + 0 + 0 + \alpha(1+\theta) = 0.
\end{gather*}
 It follows that $1\not\in \langle \mu^*\so(7)\rangle$, and so $\Cliff(8) // \Spin(7) \not\cong 0$, completing the proof.
\end{proof}

\subsection*{Acknowledgements}

I would like to thank the referees for their comments and improvements to this paper. This work was completed during the ``Gone Fishing 2016'' conference at University of Colorado, Boulder, which was supported by the NSF grant DMS-1543812. This research was also supported by the NSF grant DMS-1304054. The Perimeter Institute for Theoretical Physics is supported by the Government of Canada through the Department of Innovation, Science and Economic Development Canada and by the Province of Ontario through the Ministry of Research, Innovation and Science.

\pdfbookmark[1]{References}{ref}
\LastPageEnding

\end{document}